\documentclass[11pt,reqno]{amsart}

\usepackage{packages}

\title[On some \(k\)-fold generalizations of Lovász theta
  and their sandwich theorems%
  ]{%
    On some \(k\)-fold generalizations of Lovász theta\\
    and their sandwich theorems}
\author[M.K. Carli Silva]{Marcel {K.} Carli Silva\textsuperscript{1\textasteriskcentered}\(^\dagger\)}
\address{%
  \textsuperscript{1}Institute of Mathematics, Statistics, and
  Computer Science, University of São Paulo%
}
\thanks{%
  \textsuperscript{\textasteriskcentered}%
  {M.K.} Carli Silva and {G.} Coutinho acknowledge support
  from National Council for Scientific and Technological Development
  -- CNPq (408180/2023-4).
}
\thanks{%
  \(^\dagger\)%
  {M.K.} Carli Silva acknowledges support from São Paulo Research
  Foundation (FAPESP), Brazil.
  Process Number \mbox{\#2023/03167-5}.
}
\author[G. Coutinho]{Gabriel Coutinho\textsuperscript{2\textasteriskcentered}\(^\ddagger\)}
\address{%
  \textsuperscript{2}Federal University of Minas Gerais%
}
\thanks{%
  \(^\ddagger\)%
  {G.} Coutinho acknowledges support from the Fundação de Amparo à Pesquisa do Estado de Minas Gerais (FAPEMIG)
}
\author[T. Oliveira]{Thiago Oliveira\textsuperscript{3\S}}
\address{%
  \textsuperscript{3}Georgia Institute of Technology%
}
\thanks{%
  \(^\S\)%
  {T.} Oliveira was partially supported by CAPES during the early stages
  of this work.
}
\author[L. Tunçel]{Levent Tunçel\textsuperscript{4\P}}
\address{%
  \textsuperscript{4}Department of Combinatorics and Optimization, University of Waterloo%
}
\thanks{%
  \(^\P\)%
  {L.} Tunçel acknowledges support from the Natural Sciences and
  Engineering Research Council (NSERC) of Canada via Discovery Grants.
}
\thanks{%
  \(^\ddagger\)%
  Corresponding Author: Gabriel Coutinho.
  Affiliation: Federal University of Minas Gerais.
  E-mail: \texttt{gabriel@dcc.ufmg.br}%
}

\begin{document}

\begin{abstract}
  We study several \(k\)-fold generalizations of the Lovász theta
  function associated with the maximum \(k\)-colorable induced subgraph
  problem.
  The first is the Narasimhan--Manber parameter \(\vartheta_k\).
  We prove that, for graphs whose adjacency matrix belongs to a
  homogeneous partially coherent algebra, this parameter is recovered by
  the theta number of the Cartesian product with the complete graph on
  \(k\) vertices.
  This class includes distance-regular and \(1\)-walk-regular graphs,
  and thus our result generalizes a theorem by
  \citeauthor{SinjorgoS22a}
  (\citeyear{SinjorgoS22a}) for graphs that are vertex- and
  edge-transitive.
  We introduce a new parameter \(\varphi_k\) obtained from orthonormal
  representations of graphs and show the inequality
  \(\varphi_k \leq \vartheta_k\).
  For both parameters, we study the smallest \(k\) for which the
  parameter is equal to the number of vertices; these saturation
  parameters yield lower bounds on the chromatic number.
  We determine which vertex-weighted versions of these parameters are
  gauges, and discuss a natural definition for the \(k\)-fold theta body
  of a graph.
  We conclude with open questions comparing \(\vartheta_k\),
  \(\varphi_k\), \(\vartheta(G\Cartesian K_k)\), and related
  convexifications.
\end{abstract}

\maketitle

\section{Introduction}
\label{sec:intro}

The Lovász theta number~\(\vartheta(G)\) of a graph~\(G\), introduced
in~\cite{Lovasz79a}, has remarkable applications across
combinatorial optimization and spectral graph
theory~\cite{GroetschelLS86a,GvozdenovicL08a,GodsilRRSV20a}.
Perhaps one of its most famous properties is the sandwich theorem
(see~\cite{Knuth94a}), which states that
\begin{equation}
  \label{eq:vanilla-sandwich}
  \alpha(G) \leq \vartheta(G)
  \leq \chi_f(\overline{G}) \leq \chi(\overline{G}),
\end{equation}
where \(\alpha\) denotes the independence number, \(\chi_f\) is the
fractional chromatic number, \(\chi\) is the chromatic number, and
\(\overline{G}\) is the complement of~\(G\).
Moreover, \(\vartheta(G)\) can be formulated as the optimal
value of a semidefinite
program~(SDP) which allows it to be computed in polynomial time to
arbitrary precision.

Let \(G\) be a graph and let \(k \geq 1\).
We denote by \(\alpha_k(G)\) the largest number of vertices of an
induced subgraph of~\(G\) that is \(k\)-colorable.
The \(k\)-multicoloring number \(\chi_k(G)\) of~\(G\), sometimes
called~\cite{CampeloMS16a} the \(k\)-fold colouring number of~\(G\),
is the smallest number of colors needed so that one can assign to each
vertex a \(k\)-subset of these colors and adjacent vertices are
assigned disjoint sets.
Thus, \(\alpha_1 = \alpha\) and \(\chi_1 = \chi\).
We informally refer to such parameters as \emph{\(k\)-fold
  generalizations}.
The main topic of this paper is a \(k\)-fold generalization
of~\(\vartheta\), denoted \(\vartheta_k\) and introduced by
\nameandcite{NarasimhanM88a}, which recovers \(\vartheta\)
when \(k=1\) and satisfies the following sandwich theorem:
\begin{equation}
  \label{eq:k-sandwich}
  \alpha_k(G) \leq \vartheta_k(G) \leq \chi_k(\overline{G}).
\end{equation}
Moreover, \(\vartheta_k(G)\) also admits a formulation as the optimal
value of an SDP; see, e.g.,~\cite[Sec.~5.3]{Alizadeh95a}.

A key motivation for this paper and other recent work
on~\(\vartheta_k\) is to determine the quality of approximations of
\(\alpha_k\) via geometric strengthenings of \(\vartheta_k\), and
subsequently the quality of bounds for~\(\chi\) that can be obtained
from~\(\vartheta_k\).
We survey below some of the results that motivate this paper.

Let \(G\) be a graph on \(n\) vertices.
\Nameandcite{MoharP93a} pointed out that if \(k \geq \chi(G)\), then
\(n \leq \alpha_k(G)\), whence \(n \leq \vartheta_k(G)\).
This elementary remark shows that \(\vartheta_k\) can be used to
obtain lower bounds on \(\chi\); namely,
\begin{equation}
  \label{eq:mu-def}
  \mu(G) \coloneqq \min\setst{k}{\vartheta_k(G) = n} \leq \chi(G).
\end{equation}
The recent work \cite{CarliCG21a} has shown a strong
generalization of this result by exploring vertex weights and by
replacing \(\chi(G)\) by \(\chi_f(G)\) in~\cref{eq:mu-def}.
In this paper we will introduce an SDP formulation for \(\mu(G)\),
study its behavior in highly regular graphs, and present a conjecture
relating it to \(\vartheta(\overline{G}).\)

Recent works~\cite{GvozdenovicL08a,BenedettoCC21a} provide other
frameworks for converting upper bounds for \(\alpha\) into lower
bounds for~\(\chi\).
In the approach of~\cite{GvozdenovicL08a}, the input graph parameter
is applied to the Cartesian product \(G \Cartesian K_k\).
Motivated by this, \nameandcite{SinjorgoS22a} compared
\(\vartheta(G \Cartesian K_k)\) and~\(\vartheta_k(G)\) in order to
determine which of \(\vartheta\) or \(\vartheta_k\) yields a tighter
bound for~\(\chi\).
Numerical results from~\cite{KuryatnikovaSV22a} led them to
propose the following conjecture:
\begin{conjecture}[{\cite[Conjecture~1]{SinjorgoS22a}}]
  \label{conj:ss}
  For any graph~\(G\) and integer number \(k \geq 1\), one has
  \(\vartheta(G \Cartesian K_k) \leq \vartheta_k(G)\).
  Equality holds if \(\vartheta_k(G) = k\vartheta(G)\).
\end{conjecture}

\Citeauthor{SinjorgoS22a} provided further support for the
conjecture by proving the following result and by computing the
(common) values of these two parameters for various classical
algebraic families of graphs.
\begin{theorem}[{\cite[Theorem~5.11]{SinjorgoS22a}}]
  \label{thm:ss-very-symmetric}
  Let \(G\) be a graph with \(n\) vertices that is both vertex- and
  edge-transitive.
  Let \(k \geq 1\) be an integer.
  Then
  \(\vartheta(G \Cartesian K_k) = \min\{k\vartheta(G),n\} =
  \vartheta_k(G)\).
\end{theorem}

We will extend this result to a much larger class of graphs that
include distance-regular and \(1\)-walk-regular graphs.

The paper \cite{KuryatnikovaSV22a} presents several lifted
formulations that attempt to approximate~\(\vartheta_k\).
We identify that one of them suggests the definition of \(k\)-fold
theta body, which we formalize in this paper.
This led us to study whether weighted \(\vartheta_k\) is a monotone
gauge (it is not), and to introduce several variants based on the
standard theory for \(k = 1\).
We study their behavior with regards to convexity and their relation
to \(\chi\) and \(\chi_f\).
One of the objects we introduce has an interesting connection to the
minimum dimension of orthonormal representations.
We also prove several inequalities, including examples that show some
of them are strict.

All the works mentioned above consider only integer values of~\(k\).
A subtle point of this work is that allowing fractional/real values
of~\(k\) can lead to simpler and more natural statements, smoothing
out the rounding phenomena that arise when \(k\) is restricted to
integers.

\subsection*{Our contributions}

Our main contributions are the substantial extension of
\cref{thm:ss-very-symmetric}, and the introduction of the parameter
\(\varphi_k(G,w)\), along with the proof of several of its properties.

\begin{itemize}[leftmargin=*,itemsep=8pt]
\item In Section~\ref{sec:highlyregular} we recall the definition of a
  homogeneous partially coherent algebra, and discuss several
  consequences that it implies for the optimal solutions
  determining~\(\vartheta\).
  In \cref{sec:cartesian}, we generalize \cref{thm:ss-very-symmetric}
  to the class of graphs whose adjacency matrix belongs to these
  algebras, which includes vertex-transitive graphs, distance-regular
  graphs (and thus strongly regular graphs), and 1-walk-regular
  graphs.
  This represents a wide extension of graphs for which
  Conjecture~\ref{conj:ss} is known to hold, and it provides a unified
  approach to several families of graphs considered
  in~\cite[Sections~5 and~6]{SinjorgoS22a}.
  In \cref{sec:mu}, we show that \(\mu(G) = \vartheta(\overline{G})\)
  for graphs whose adjacency matrix defines a homogeneous partially
  coherent algebra, and we propose the conjecture that
  \(\mu(G) \leq \vartheta(\overline{G})\) for all graphs.
\item In \cref{sec:varphi}, we define and study the graph parameter
  \(\varphi_k\) inspired by one of the many alternative formulations
  of \(\vartheta\) given by \nameandcite{Lovasz79a}, based on
  orthonormal representations.
  We show that, although \(\varphi_1 = \vartheta_1\), for other values
  of~\(k\) one only has the inequality
  \(\varphi_k(G) \leq \vartheta_k(G)\).
  We relate \(\varphi_k\) to the minimum dimension
  \(\xi(\overline{G})\) of an orthonormal representation
  of~\(\overline{G}\): we show that for
  \(k \coloneqq \xi(\overline{G})\) one has \(\varphi_k(G) = n\), the
  number of vertices of~\(G\).
\item In \cref{sec:alphak-prime}, we introduce \(\alpha_k'(G)\) to be
  the largest number of vertices in an induced subgraph of \(G\) whose
  fractional chromatic number is \(\leq k\).
  We show that \(\alpha_k'(G) \leq \vartheta_k(G)\), so that
  \(\vartheta_k\) cannot distinguish
  \(k\)\nobreakdash-colorability from having \emph{fractional}
  chromatic number bounded by~\(k\).
\item In \cref{sec:weights}, we extend several parameters to
  vertex-weighted versions, preserving the known inequalities.
  We show that, if \(w \in \Reals_+^V\) is a vector of nonnegative
  vertex weights, the weighted variation \(\vartheta_k(G,w)\) of
  \(\vartheta_k(G)\) is \textbf{not} sublinear, while
  \(\varphi_k(G,w)\) is.
  In particular, unlike the \(k=1\) case, \(\vartheta_k(G,w)\) cannot
  in general be described as the support function of a convex body
  analogous to the theta body.
\item In Section~\ref{subsec:psi-k}, we recall the vertex-weighted
  version of the \(\psi_k\) studied by Lovász in the context of
  \(k\)-perfect graphs, and we show that
  \(\varphi_k(G,w) \leq \psi_k(\overline{G},w)\).
  Thus, \(\varphi_k\) is a variant that satisfies a sandwich
  inequality conjectured by
  \citeauthor{Narasimhan89a}~\cite[Sec.~4.4.3]{Narasimhan89a} for
  \(\vartheta_k\).
\item In \cref{sec:theta-body-k}, we introduce
  \[
    \THbody_k(G)\coloneqq k\THbody(G)\cap[0,1]^V
  \]
  and its support function \(\theta_k(G,w)\).
  We give a semidefinite representation for \(\THbody_k(G)\),
  prove \(\alpha_k(G,w)\leq \varphi_k(G,w) \leq \theta_k(G,w)\), and
  discuss how this construction relates to matrix-lifting relaxations
  of~\cite{KuryatnikovaSV22a}.
\end{itemize}

\subsection*{Notation and background on \texorpdfstring{\(\vartheta_k\)}{ϑₖ}}

We typically denote a graph as \(G = (V,E)\), with
\(n \coloneqq \card{V}\).
We write \(\Sym{V}\) for the set of real symmetric \(V \times V\)
matrices, and \(\Psd{V}\) for the subset of these which are positive
semidefinite.
If \(M \in \Sym{V}\), then we shall typically order its eigenvalues as
\begin{equation*}
  \lambda_1(M)\ge\cdots\ge\lambda_n(M).
\end{equation*}
The Schur product of matrices \(A\) and \(B\) of same order is defined
by
\[(A \schur B)_{ij} = A_{ij}B_{ij}.\]
The (trace) inner product of \(A,B \in \Sym{V}\) is \(\ip{A}{B}=\trace(AB)\).
It is convenient to observe that
\[\ip{A}{B} = \ip{J}{A \schur B},\]
where \(J = \ones\ones^\T\) is the matrix of all-ones and \(\ones\) is
the vector of all-ones.
That~is, \(\ip{A}{B}\) is the sum of all entries of the matrix
\(A \schur B\).
The operator \(\diag\) maps a \(V \times V\) matrix \(Y\) to a vector in
\(\Reals^V\) consisting of the diagonal entries of \(Y\).

For each \(k\in[n] \coloneqq \{1,\dots,n\}\) and \(M \in \Sym{V}\), denote
the Ky Fan \(k\)-sum
\begin{equation*}
  \Lambda_k(M)
  \coloneqq
  \sum_{i=1}^k\lambda_i(M).
\end{equation*}
Recall that \(\Lambda_k(M)\) is the optimal value of the following
primal-dual pair of SDPs (see \cite{Fan49a,NesterovN94a,Alizadeh95a}):
\begin{equation}
  \label{eq:kyfan}
  \begin{aligned}
    \Lambda_k(M)
    & =
    \max\setmap[\big]{\ip{M}{X}}{0\preceq X\preceq I,\,\ip{I}{X}=k}\\
    & =
    \min\setmap[\big]{k\nu + \ip{I}{Y}}%
      {\nu \in \Reals,\, Y \in \Psd{V},\, M\preceq Y+\nu{}I};
  \end{aligned}
\end{equation}
here we write \(A \preceq B\) to mean that \(B-A\) is positive
semidefinite.
Using these formulations, \(\Lambda_k\) may be defined for \emph{real}
nonnegative \(k\).

Denote the adjacency matrix of a graph~\(G\) as~\(A_G\).
For each \(k \in [n]\), the graph parameter \(\vartheta_k\) is defined
by
\begin{equation}
  \label{eq:thkdef}
  \vartheta_k(G)
  \coloneqq
  \max\setmap[\big]{\ip{J}{X}}{
    \ip{I}{X} = k,\, A_G \schur X = 0,\, 0 \preceq X \preceq I
  }.
\end{equation}
When \(k = 1\), the constraint \(X \preceq I\) is implied by the
constraints \(X \succeq 0\) and \(\ip{I}{X} = 1\), so
\(\vartheta_1(G)\) recovers the famous Lovász theta
number~\(\vartheta(G)\):
\begin{equation}
  \label{eq:thdef}
  \vartheta(G)
  =
  \max\setmap[\big]{\ip{J}{X}}{
    \ip{I}{X} = 1,\, A_G \schur X = 0,\, X \succeq 0
  }.
\end{equation}
Note that the SDP in~\cref{eq:thkdef} is perfectly well defined for
each~\(k\) in the \emph{real} interval \([1,n]\).
We adopt this extended definition of~\(\vartheta_k\).
Furthermore, the optimal value in~\cref{eq:thkdef} does not change if
one replaces the constraint \(\ip{I}{X} = k\) with
\(\ip{I}{X} \leq k\), which shows that \(\vartheta_k(G)\) is a
nondecreasing function of~\(k\).

Strong duality holds (see \cite{Alizadeh95a} or \cite[Section
3.1]{KuryatnikovaSV22a}), and therefore
\[ \vartheta_k(G) = \min \left\{\, k\nu  + \ip{I}{Y} :
    \nu \in \Reals, \, Y \in \Psd{V}, \,
    Z \in \Sym{V}, \, Z \schur A_G = Z, \,
    \nu{}I + Y + Z \succcurlyeq J \,\right\}.
\]

\section{\texorpdfstring{\(\vartheta_k\)}{ϑₖ} for highly regular graphs}
\label{sec:highlyregular}

In this section we explore two aspects of \(\vartheta_k\) for graphs
admitting a strong form of regularity, which we introduce below.
For these graphs, we significantly extend a result of
\nameandcite{SinjorgoS22a} comparing \(\vartheta_k(G)\) and
\(\vartheta(G \Cartesian K_k)\); and we study the chromatic number
lower bound \(\mu(G)\) defined in~\cref{eq:mu-def}.

We follow a definition from \cite{MancinskaRV24a}: for a graph
\(G\) with adjacency matrix \(A_G\), we say that a sub-algebra
\(\mathcal{A}_G\) of \(\Reals^{V \times V}\) is a partially coherent
algebra with respect to \(I\) and \(A_G\) if it contains \(I\), \(J\)
and \(A_G\), is self-adjoint, and is closed under Schur product with
\(I\) and \(A_G\).
We say that \(\mathcal{A}_G\) is homogeneous if every matrix in it has
constant diagonal.

The following are all examples of graphs whose adjacency matrix
belongs to a homogeneous partially coherent algebra: distance-regular
graphs, vertex-transitive graphs, or more generally graphs whose
adjacency matrix belongs to homogeneous coherent algebras; and
\(1\)-walk-regular graphs.
In \cite{CarliCGR19a}, some sufficient conditions for
\(\vartheta(G) \overline{\vartheta}(G) = n\) were introduced, and they
are satisfied by graphs whose adjacency matrix belongs to a
homogeneous partially coherent algebra.
We summarize the main result we need below.

\begin{theorem}[Corollary 4.2 in \cite{CarliCGR19a}.]
  \label{thm:ththbar}
  Let \(G\) be a graph on \(n\) vertices whose adjacency matrix
  belongs to a homogeneous partially coherent algebra.
  Then
  \[
    \vartheta(G) \vartheta(\overline{G}) = n.
  \]
\end{theorem}

We may also obtain important information about the largest eigenvalue
of optimal solutions.

\begin{lemma}
  \label{lem:largestevalue}
  Let \(G\) be a graph on \(n\) vertices.
  If there exists an optimal solution \(\hat{X}\) for
  \(\vartheta \coloneqq\vartheta(G)\) in \cref{eq:thdef} with constant
  diagonal, then
  \begin{equation}
    \label{eq:X-lambdamax}
    \hat{X} \preceq \dfrac{\vartheta}{n} I,
  \end{equation}
  and
  \begin{equation}
    \label{eq:X-lambdamax-eq}
    \hat{X} \ones = \dfrac{\vartheta}{n} \ones.
  \end{equation}
  Moreover, if the adjacency matrix of \(G\) belongs to a homogeneous
  partially coherent algebra, then such solution exists.
\end{lemma}
\begin{proof}
  Let \(\hat{X}\) be an optimal solution for \cref{eq:thdef} with
  constant diagonal.
  Let \(h\) be a \(\lambda\)-eigenvector of \(\hat{X}\) such that
  \(\pnorm{h}=1\).
  Define
  \(\hat{Y} \coloneqq n(\hat{X} \schur hh^{\transp}) \succeq 0\).
  We show that \(\hat{Y}\) is feasible in \cref{eq:thdef}.
  If \(ij\) is an edge, then
  \(\hat{Y}_{ij} = n h_i h_j \hat{X}_{ij} = 0\).
  Since \(\hat{X}\) has constant diagonal, we have
  \(\diag(\hat{X}) = \tfrac{1}{n}\ones\), so
  \(\diag(\hat{Y}) = h \schur h\).
  Thus, \(\trace(\hat{Y}) = \pnorm{h}^2 = 1\).
  Now \(\hat{X}\) is optimal in \cref{eq:thdef}, so a comparison of
  the objective values of \(\hat{X}\) and \(\hat{Y}\) yields
  \begin{equation*}
    \vartheta
    =
    \iprod{J}{\hat{X}}
    \geq
    \iprod{J}{\hat{Y}} =
    n \iprod{J}{\hat{X} \schur hh^{\transp}}
    =
    n \iprod{\hat{X}}{hh^{\transp}}
    =
    n h^{\transp}\hat{X}h
    =
    n\lambda.
  \end{equation*}
  Thus, \cref{eq:X-lambdamax} is proved.
  Moreover, the Rayleigh quotient characterization of the largest
  eigenvalue of \(\hat{X}\) gives
  \[ \vartheta = \langle J , \hat{X} \rangle \leq n~
    \lambda_{\max}(\hat{X}) \leq n~\frac{\vartheta}{n} = \vartheta,\]
  thus \(\ones\) is an eigenvector for the eigenvalue
  \(\lambda_{\max}(\hat{X}) = \vartheta / n\), proving
  \eqref{eq:X-lambdamax-eq}.

  For the last part, take any optimal solution \(X'\) in
  \cref{eq:thdef}, and let \(\hat{X}\) be its orthogonal projection
  onto the homogeneous partially coherent algebra \(\mathcal{A}_G\).
  Since \(J,I \in \mathcal{A}_G\), we have
  \(\iprod{J}{X'} = \iprod{J}{\hat{X}}\) and
  \(\iprod{I}{X'} = \iprod{I}{\hat{X}}\).
  That \(\hat{X} \succeq 0\) follows from the standard fact that the
  projection of a positive semidefinite matrix onto a \(*\)-algebra is
  positive semidefinite; see for instance the discussion in
  \cite[Section 2]{CarliCGR19a}.
  It remains to show that \(A_G \circ \hat{X} = 0\).
  There is an orthonormal basis \(A_0,\dotsc,A_d\) of
  \(\mathcal{A}_G\) and an index \(k \in [d]\) such that \(A_0\) is a
  multiple of~\(I\), the support of each matrix in \(A_1,\dotsc,A_k\)
  is contained in the support of \(A_G\), and the support of each
  matrix in \(A_{k+1},\dotsc,A_d\) is contained in the support of
  \(A_{\overline{G}}\).
  Since \(A_G \circ X' = 0\), we have \(A_i \circ X' = 0\) and thus
  \(\ip{A_i}{X'} = 0\) for each \(i \in [k]\).
  Thus, some terms in the projection
  \(\hat{X} = \sum_{i=0}^d \ip{A_i}{X'} A_i\) can be dropped:
  \(\hat{X} = \ip{A_0}{X'} A_0 + \sum_{i=k+1}^d \ip{A_i}{X'} A_i\).
  This shows that \(A_G \circ \hat{X} = 0\).
\end{proof}

\subsection{The theta of a cartesian product}
\label{sec:cartesian}

We start this section addressing a subtle point.
In its original definition, \(\vartheta_k\) is presented only for
\(k \in \Integers_+\).
However, definition \eqref{eq:thkdef} is immediately extended for
\(k \in \Reals_+\), and we shall adopt this direction in this paper.
In fact, \citeauthor{SinjorgoS22a}
\cite[Section~3]{SinjorgoS22a} prove that
\(\vartheta_k(G) \leq n\) for all \(k\) (integer), and that
\(\vartheta_k(G) \leq \vartheta_{k+1}(G)\), with equality if and only
if \(\vartheta_k(G) = n\).
Carli Silva et al.~\cite[Lemma 1]{CarliCG21a} showed the same
holds for real \(k\).

\Nameandcite{GvozdenovicL08a} derive lower bounds on~\(\chi\)
from~\(\vartheta\) by computing \(\vartheta(G \Cartesian K_k)\) for
each integer \(k \geq 1\).
While doing so, they obtain \(\vartheta(G \Cartesian K_k)\) as the
optimal value of the following SDP
(see~\cite[Theorem~2.7]{GvozdenovicL08a}):
\begin{equation}
  \label{eq:theta-prod-def}
  \begin{alignedat}[t]{3}
    \vartheta(G \Cartesian K_k)
    =
    \text{Maximize \ } & k\ip{J}{X} + k(k-1)\ip{J}{Y},
    \\
    \text{subject to \ } & X,\, Y \in \Sym{V},
    \\
    & X \succeq Y,
    \\
    & X + (k-1)Y \succeq 0,
    \\
    & A_G \schur X = 0,
    \\
    & I \schur Y = 0,
    \\
    & k\ip{I}{X} = 1.
  \end{alignedat}
\end{equation}
This SDP is well defined for any real \(k \in [1,n]\).
We will adopt the notation \(\vartheta(G \Cartesian K_k)\) as
in~\cref{eq:theta-prod-def} for any real \(k \in [1,n]\), even though
there is no corresponding complete graph~\(K_k\) if \(k\) is not an
integer.

\Nameandcite{SinjorgoS22a}, while applying the approach by
\nameandcite{GvozdenovicL08a} to lower bound the chromatic number,
compared the parameter \(\vartheta(G\Cartesian K_k)\) with the
Narasimhan--Manber parameter \(\vartheta_k(G)\).
Here we extend a result they showed for graphs which are vertex- and
edge-transitive.

\begin{theorem}
  \label{thm:2}
  Let \(G\) be a graph on \(n\) vertices whose adjacency matrix
  belongs to a homogeneous partially coherent algebra, and let
  \(k \geq 1\) be a real number.
  Then
  \begin{equation}
    \label{eq:1}
    \vartheta(G \Cartesian K_k) = \min\{k\vartheta(G),n\} =
    \vartheta_k(G).
  \end{equation}
\end{theorem}

\begin{proof}
  Throughout this proof we denote \(\vartheta \coloneqq \vartheta(G)\)
  and \(\overline{\vartheta} \coloneqq \vartheta(\overline{G})\).

  We first prove that
  \begin{equation}
    \label{eq:main-thetak}
    \vartheta_k(G) = \min\{k\vartheta(G),n\},
  \end{equation}
  starting with `\(\leq\)'.
  If \(X\) is feasible in \cref{eq:thkdef}, then \(X \preceq I\) so
  its objective value is \(\ip{J}{X} \leq \ip{J}{I} = n\).
  Moreover, \(\tfrac{1}{k}X\) is feasible in \cref{eq:thdef} with
  objective value \(\tfrac{1}{k}\ip{J}{X}\).
  This shows that \(\vartheta \geq \tfrac{1}{k}\vartheta_k(G)\).
  Now we prove `\(\geq\)'.
  By Lemma~\ref{lem:largestevalue}, there is an optimal solution
  \(\hat{X}\) for \cref{eq:thdef} with constant diagonal.
  We will show that
  \(\hat{Z} \coloneqq \min\{k,\overline{\vartheta}\}~\hat{X}\) is
  feasible in~\cref{eq:thkdef} having as its objective value the RHS
  of~\cref{eq:main-thetak}.
  From Theorem~\ref{thm:ththbar} we have
  \(n = \vartheta\overline{\vartheta}\), so \cref{eq:X-lambdamax}
  shows that \(\hat{X} \preceq (1/\overline{\vartheta}) I\).
  Hence, \(\hat{Z} \preceq \overline{\vartheta} \hat{X} \preceq I\),
  so \(\hat{Z}\) is feasible in~\cref{eq:thkdef} with trace constraint
  relaxed to \(\leq k\).
  Its objective value is
  \(\iprod{J}{\hat{Z}} =
  \min\{k,\overline{\vartheta}\}\iprod{J}{\hat{X}} =
  \min\{k,\overline{\vartheta}\}\,\vartheta = \min\{k\vartheta,n\}\).
  This concludes the proof of~\cref{eq:main-thetak}.

  Now we prove
  \begin{equation}
    \label{eq:main-theta-prod}
    \vartheta(G \Cartesian K_k) = \min\{k\vartheta(G),n\},
  \end{equation}
  again starting with `\(\leq\)'.
  Let \((\hat{X},\hat{Y})\) be feasible in~\cref{eq:theta-prod-def}.
  Note that \(\hat{X}\succeq0\) follows from
  \(\hat{X}+(k-1)\hat{Y}\succeq0\) and \(\hat{X}-\hat{Y}\succeq0\) via
  \(\hat{X} = \tfrac{1}{k}(\hat{X}+(k-1)\hat{Y} +
  (k-1)(\hat{X}-\hat{Y}))\).
  Thus \(k\hat{X}\) is feasible in~\cref{eq:thdef}, so the objective
  value of \((\hat{X},\hat{Y})\) is
  \begin{equation*}
    k\iprod{J}{\hat{X}} + k(k-1)\iprod{J}{\hat{Y}}
    \leq
    k\iprod{J}{\hat{X}} + k(k-1)\iprod{J}{\hat{X}}
    =
    k\iprod{J}{k\hat{X}}
    \leq
    k\vartheta.
  \end{equation*}
  Expanding the inner product
  \(\iprod[\big]{\hat{X}+(k-1)\hat{Y}}{k(nI-J)} \geq 0\) while using
  \(\iprod{I}{\hat{Y}} = 0\) and \(k\iprod{I}{\hat{X}} = 1\) shows
  that the objective value of \((\hat{X},\hat{Y})\) is \(\leq n\).

  It remains to prove `\(\geq\)' in~\cref{eq:main-theta-prod}.
  Let \(\hat{X}\) be an optimal solution for \(\vartheta\) in
  \cref{eq:thdef} with constant diagonal, which exists by
  Lemma~\ref{lem:largestevalue}.
  If \(k=1\), then \cref{eq:theta-prod-def} reduces to the usual
  SDP~\cref{eq:thdef} for \(\vartheta(G)\), so assume \(k>1\).
  Let \(\lambda \coloneqq \max\{k , \overline{\vartheta}\}>1\).
  Define
  \[
    X_0 \coloneqq \frac1k \hat{X},
    \qquad
    Y_0 \coloneqq
    \frac{1}{k(\lambda-1)}
    \left(\frac1n J-\hat{X}\right).
  \]
  We claim that \((X_0,Y_0)\) is feasible in~\cref{eq:theta-prod-def},
  with objective value \(\min\{k \vartheta , n \}\).
  The constraints \(A_G \circ X_0 = 0\), \(I \circ Y_0 = 0\), and
  \(k \langle I , X_0 \rangle = 1\) follow immediately from the
  definition of \(\hat X\).
  It remains to check the two semidefinite constraints.
  We have
  \[
    X_0-Y_0
    =
    \frac{1}{k(\lambda-1)}
    \left(\lambda\hat{X}-\frac1n J\right).
  \]
  In the orthogonal complement of \(\ones\) this is clearly positive
  semidefinite, and \eqref{eq:X-lambdamax-eq} implies that the vector
  \(\ones\) is an eigenvector of \(\lambda\hat{X}-\tfrac{1}{n}J\) with
  eigenvalue
  \[
    \frac{\lambda\vartheta}{n}-1,
  \]
  which is nonnegative because
  \[
    \lambda\geq\overline{\vartheta}=\frac{n}{\vartheta}.
  \]
  Thus \(X_0-Y_0 \succeq 0\). Similarly,
  \[
    X_0+(k-1)Y_0 =
    \frac{1}{k(\lambda-1)}
    \left((\lambda-k)\hat{X} + \frac{k-1}{n}J \right).
  \]
  Since \(\lambda\geq k\), and since both \(\hat{X}\) and \(J\) are
  positive semidefinite, we get
  \[
    X_0+(k-1)Y_0 \succeq 0.
  \]
  We now compute its objective value:
  \begin{align*}
    k\iprod{J}{X_0} + k(k-1)\iprod{J}{Y_0}
    &= \iprod{J}{\hat{X}} + \frac{k-1}{\lambda-1}
    \iprod[\Big]{J}{\frac1nJ-\hat{X}}
    = \vartheta + \frac{k-1}{\lambda-1}(n-\vartheta).
  \end{align*}
  If \(\lambda=\overline{\vartheta}=n/\vartheta\), then
  \[ \vartheta +\frac{k-1}{\overline{\vartheta}-1}(n-\vartheta) =
    \vartheta + (k-1)\vartheta
    = k\vartheta.
  \]
  If \(\lambda=k\), then
  \[
    \vartheta + \frac{k-1}{k-1}(n-\vartheta)
    = n.
  \]
  Since \(\lambda=\max\{k,\overline{\vartheta}\}\), these two cases
  give exactly
  \[
    k\iprod{J}{X_0} + k(k-1)\iprod{J}{Y_0}
    = \min\{k\vartheta,n\}.
  \]
  Therefore
  \[
    \vartheta(G\Cartesian K_k)\geq \min\{k\vartheta,n\}.
  \]
  Together with the upper bound already proved, this gives
  \[
    \vartheta(G\Cartesian K_k)=\min\{k\vartheta(G),n\}.\qedhere
  \]
\end{proof}

\subsection{The \texorpdfstring{\(\mu\)}{μ} parameter}
\label{sec:mu}

Let \(G\) be a graph on \(n\) vertices.
Define
\begin{equation}
  \label{eq:4}
  \mu(G) \coloneqq \min\setst[\big]{k \in [1,n]}{\vartheta_k(G) = n}.
\end{equation}
Note that~\cite[Equation~(39)]{SinjorgoS22a} defines
\[
  \Psi_{\vartheta_k}(G)
  =
  \min\setst{k \in \Naturals}{\vartheta_k(G) = n};
\]
the parameter \(\mu(G)\) is the continuous extension of this, with
\(\ceiling{\mu(G)} = \Psi_{\vartheta_k}(G)\).

From \(\alpha_k(G) \leq \vartheta_k(G)\), it follows that
\[
  \mu(G) \leq \chi(G).
\]
The well known fact that \(\vartheta(\overline{G}) \leq \chi(G)\) motivates
the question:
\[\text{How do \(\mu(G)\) and \(\vartheta(\overline{G})\) compare?}\]
In order to address this question, we first introduce a formulation of
\(\mu(G)\) as an SDP, which follows immediately from
\eqref{eq:thkdef} and relies crucially on allowing a continuous range
for \(k\) in \([1,n]\):
\begin{equation}
  \label{eq:mudef}
  \mu(G) = \min\Big\{\, \ip{I}{X} : 0 \preccurlyeq X \preccurlyeq I,\,
  A_G \circ X = 0, \, \ip{J}{X} = n \,\Big\}.
\end{equation}

\begin{theorem}
  \label{thm:mu}
  Let \(G = (V,E)\) be a graph on \(n\) vertices.
  If there exists an optimal solution for
  \(\vartheta(G)\) in \cref{eq:thdef} with constant
  diagonal, then
  \[
    \mu(G) = \frac{n}{\vartheta(G)}.
  \]
\end{theorem}
\begin{proof}
  Let \(\hat{X}\) be an optimal solution for
  \(\vartheta \coloneqq\vartheta(G)\) in \cref{eq:thdef} with constant
  diagonal.
  From Lemma~\ref{lem:largestevalue} it follows that
  \((n / \vartheta) \hat{X}\) is feasible for \eqref{eq:mudef} with
  objective value \(n/\vartheta\), thus
  \[
    \mu(G) \leq \frac{n}{\vartheta}.
  \]
  On the other hand, it is known (and easy to see) that
  \begin{equation}
    \label{eq:2}
    \frac{1}{\vartheta} =  \min \Big\{\, \ip{I}{Y} : \ip{J}{Y} = 1,\,
      A_G \circ Y = 0, \, Y \succeq 0 \, \Big\}.
  \end{equation}
  Thus, if \(X\) is feasible in~\cref{eq:mudef}, then
  \(\tfrac{1}{n}X\) is feasible in~\cref{eq:2} with objective value
  \(\tfrac{1}{n}\ip{I}{X}\).
  Hence
  \[
    \frac{1}{\vartheta} \leq \frac{\mu(G)}{n}.\qedhere
  \]
\end{proof}

The second part of the proof above shows that, for any graph~\(G\) on
\(n\) vertices,
\begin{equation}
  \label{eq:6}
  \vartheta(G)\mu(G) \geq n.
\end{equation}
\begin{corollary}
  If \(G\) is a graph whose adjacency matrix belongs to a homogeneous
  partially coherent algebra, then
  \[\mu(G) = \vartheta(\overline{G}). \]
\end{corollary}
\begin{proof}
  Immediate from the result above, with Theorem~\ref{thm:ththbar} and
  Lemma~\ref{lem:largestevalue}.
\end{proof}

Computations by Pedro Cipriano first suggested that there are graphs
for which \(\mu(G) \neq \vartheta(\overline{G})\).
Indeed, an explicit example is obtained by taking \(G\) to be the
disjoint union \(K_2 \cup K_1\), so that \(\overline G = K_{1,2}\).
Thus \(\vartheta(\overline G) = 2\).
However, since
\[
  X\coloneqq\frac13
  \begin{bmatrix}
    2 & 0 & 1\\
    0 & 2 & 1\\
    1 & 1 & 1
  \end{bmatrix}
\]
is feasible in~\cref{eq:mudef} with objective value \(5/3\), we obtain
\(\mu(G) \leq 5/3 < 2 = \vartheta(\overline G)\).

We were not able to find an example in which
\(\mu(G) > \vartheta(\overline{G})\), thus we are willing to conjecture that
for all graphs, \(\mu(G) \leq \vartheta(\overline{G})\).

Using~\cref{eq:6} and the fact that
\(\vartheta(G)\vartheta(\overline G) \geq n\) for every graph \(G\) on \(n\)
vertices, if the conjecture holds, we would get
\(
  \frac{n}{\vartheta(G)} \leq \mu(G) \leq \vartheta(\overline G).
\)

\section{The \texorpdfstring{\(\varphi_k\)}{φₖ} parameter and orthonormal representations}
\label{sec:varphi}

Among the many equivalent definitions of \(\vartheta(G)\), we draw
attention to the following (see \cite[Theorem 6]{Lovasz79a})
\[
  \vartheta(G) = \max \left\{ \, \lambda_1(Y) : Y \in \Psd{V}, \, \diag
    Y = \ones, \, A_G \circ Y = 0\right\}.
\]
Motivated by this, and in the hopes of finding an equivalent
characterization of \(\vartheta_k\), we introduce the parameter below
for any positive integer \(k\):
\begin{equation}
  \label{def:varphi}
  \varphi_k(G) \coloneqq \max \left\{ \, \sum_{i = 1}^k
    \lambda_i(Y) : Y \in \Psd{V}, \, \diag Y = \ones, \,
    A_G \circ Y = 0 \right\}.
\end{equation}
By Ky Fan's characterization (see~\cref{eq:kyfan}), we may
extend this to nonnegative real~\(k\) by
\begin{equation}
  \label{def:varphi2}
  \varphi_k(G) \coloneqq \max \left\{ \, \ip{M}{Y} : M,Y \in \Psd{V},\,
    \diag Y = \ones, \, A_G \circ Y = 0, \, \ip{I}{M} = k,\,
    M \preccurlyeq I \right\}.
\end{equation}

\begin{theorem}
  \label{thm:1}
  For any graph \(G\) and positive integer \(k\),
  \[
    \alpha_k(G) \leq \varphi_k(G) \leq \vartheta_k(G).
  \]
  The second inequality also holds for \(k\) a nonnegative real.
\end{theorem}
\begin{proof}
  To see the first inequality, consider a maximum induced subgraph of
  \(G\) that is \(k\)-colorable, and define the matrix \(Y \in \Sym{V}\)
  by introducing square blocks of \(1\)s corresponding to each color
  class, padding the remaining diagonal entries also with~\(1\)s, and
  \(0\)s otherwise.
  It is immediate to see that \(Y\) is feasible for \eqref{def:varphi}
  with objective value \(\alpha_k\).

  For the second inequality, let \(k\) be a nonnegative real, and let
  \((\hat{Y},\hat{M})\) be optimal in \eqref{def:varphi2}.
  Let \(\hat{X} \coloneqq \hat{Y} \circ \hat{M}\).
  We will prove that \(\hat{X}\) is feasible for \eqref{eq:thkdef}
  with objective value \(\varphi_k(G)\).
  We have \(\iprod{I}{\hat{X}} = \iprod{I}{\hat{M}} = k\) using that
  \(\diag \hat{Y} = \ones\).
  Since \(A_G \circ \hat{Y} = 0\), we also have
  \(A_G \circ \hat{X} = 0\).
  That \(\hat{X}\) is positive semidefinite follows from a theorem of
  Schur (see \cite{Schur11a}), as both \(\hat{Y}\) and
  \(\hat{M}\) are positive semidefinite.
  Finally, \(\diag \hat{Y} = \ones\) implies that
  \[
    I - \hat{X}= I - (\hat{Y} \circ \hat{M}) = \hat{Y} \circ (I - \hat{M});
  \]
  since both of these factors are positive semidefinite, it follows
  that \(\hat{X} \preceq I\).
  Thus, \(\hat{X}\) is feasible, and its objective value is
  \[
    \ip{J}{\hat{X}} = \ip{J}{\hat{Y} \circ \hat{M}} = \ip{\hat{Y}}{\hat{M}} =
    \varphi_k(G).\qedhere
  \]
\end{proof}

We now relate \(\varphi_k\) to orthonormal representations.
An \emph{orthonormal representation} of \(G\) is a map
\(u\colon V(G)\to \mathbb{R}^m\) (for some \(m\)) such that
\(\|u_i\|=1\) for all \(i\in V(G)\) and \(\langle u_i,u_j\rangle =0\)
whenever \(ij\in E(\overline{G})\).

We provide an equivalent definition for \(\varphi_k\), following the
known correspondence between \(\vartheta = \varphi_1\) and orthonormal
representations.

\begin{theorem}
  \label{thm:varphi-orth-rep}
  Let \(G = (V,E)\) be a graph on \(n\) vertices. Then
  \[
    \varphi_k(G) = \max \left\{ \, \sum_{i \in V} u_i^\T M u_i :
      \begin{array}{l}
        u \text{ is an \(n\)-dim orthonormal representation of }\overline{G}, \\
        \ip{I}{M} = k, \, 0 \preccurlyeq M \preccurlyeq I .
      \end{array} \right\}.
  \]
\end{theorem}
\begin{proof}
  Let \((\hat{M},\hat{Y})\) be a feasible solution
  for~\cref{def:varphi2}.
  Write \(\hat{Y} = U^\T U\) for some matrix~\(U\), and set \(u_i\) to
  be \(i\)th column of~\(U\) for each \(i \in V\).
  Thus, \(u\) is an orthonormal representation of~\(\overline{G}\).
  Let \(\bar{M}\) be an optimal solution for \(\Lambda_k(U U^\T)\)
  in~\cref{eq:kyfan}.
  Since \(U U^\T\) and \(U^\T U\) have the same nonzero eigenvalues,
  \[
  \ip{\hat{M}}{\hat{Y}}
  \leq
  \Lambda_k(U^\T U)
  =
  \Lambda_k(U U^\T)
  =
  \ip{\bar{M}}{U U^\T}
  =
  \trace(U^\T \bar{M} U)
  =
  \sum_{i\in V} \qform{\bar{M}}{u_i}.
  \]
  This proves `\(\leq\)`.
  The proof of the reverse inequality is symmetric.
\end{proof}
Note that if \(k\) is a positive integer, the theorem above
immediately implies that
\[
  \varphi_k(G) = \max \left\{ \, \sum_{i \in V} u_i^\T P u_i : \begin{array}{l}
u \text{ is an \(n\)-dim orthonormal representation of }\overline{G}, \\
P \text{ is a rank-\(k\) orthogonal projector.}
\end{array} \right\}.\]

A key parameter in the study of orthonormal representations of graphs
is the minimum dimension in which such a representation exists.
Let \(\xi(G)\) denote such parameter.
It is known~\cite{Lovasz79a} that
\[
  \vartheta(G) \leq \xi(G) \leq \chi(\overline G).
\]

\begin{theorem}
  \label{thm:varphi-orth-rank}
  Let \(G\) be a graph on \(n\) vertices. Then
  \begin{equation}
    \label{eq:3}
    \xi(\overline G) = \min\{k \in \Integers_+ : \varphi_k(G) = n\}.
  \end{equation}
\end{theorem}
\begin{proof}
  We use the fact that the rank-\(d\) feasible matrices
  in~\cref{def:varphi} are precisely the Gram matrices of orthonormal
  representations of~\(\overline{G}\) that span \(d\) dimensions; see, e.g.,
  \cite[Lemma 2.2]{LaurentP96a}.
  Hence, \(\xi(\overline G)\) is the minimum rank in this set.

  Let \(Y\) be any such matrix, and write it as \(Y = U^\T U\) for
  some \([d] \times V\) matrix~\(U\), so that the columns of~\(U\)
  form an orthonormal representation of~\(\overline G\).
  Since \(\trace U^\T U = n\), for any \(k \geq 0\) we have
  \[
    \sum_{i = 1}^k \lambda_i(U^\T U) = n \iff k \geq d.\qedhere
  \]
\end{proof}

\begin{corollary}
  \label{cor:1}
  For every graph \(G\),
  \begin{equation}
    \label{eq:5}
    \mu(G) \leq \xi(\overline G).
  \end{equation}
\end{corollary}
\begin{proof}
  Let \(G\) have \(n\) vertices.
  If \(k \in \Integers_+\) is such that \(\varphi_k(G)=n\), then
  \(\vartheta_k(G)=n\) follows from \(\varphi_k(G)\leq\vartheta_k(G)\)
  in Theorem~\ref{thm:1}.
  Thus, every feasible \(k\) in the RHS of~\cref{eq:3} is feasible
  in~\cref{eq:mu-def}.
  The result now follows from Theorem~\ref{thm:varphi-orth-rank}.
\end{proof}

In the RHS of~\cref{eq:3}, one might consider allowing non-integer
\(k\) to be feasible, analogously to what we do in~\cref{eq:4}
for~\(\vartheta_k\).
However, doing so does not change the minimum value
(see~\cite[Theorem~5.12]{Oliveira24a}).

\section{\texorpdfstring{\(\vartheta_k(G)\)}{ϑₖ(G)} does not distinguish fractional chromatic number}
\label{sec:alphak-prime}

We recall the definition of the fractional chromatic number
\(\chi_f(G)\) of a graph \(G = (V,E)\).
Let \(\cS(G)\) be the set of nonempty cocliques of \(G\).
Then \(\chi_f(G)\) is the optimal value of the following linear
program (LP):
\begin{equation}
  \label{eq:7}
  \chi_f(G) \coloneqq \min\left\{\sum_{S\in\cS(G)}y_S :
    y\in\Reals_+^{\cS(G)},\ \sum_{S\in\cS(G)}y_S\ones_S = \ones\right\}.
\end{equation}
Here, \(\ones_S \in \set{0,1}^V\) denotes the incidence vector of
\(S \subseteq V\).
It is well known that \(\chi_f(G) \leq \chi(G)\).

For each real \(k \geq 1\), let \(\alpha_k'(G)\) be the largest number
of vertices in an induced subgraph of~\(G\) whose fractional chromatic
number is \(\leq k\).
Note that \(\alpha_k'(G) \geq \alpha_k(G)\).
The motivating question of this section is whether \(\vartheta_k(G)\)
can distinguish \(k\)-colorability from having fractional chromatic
number at most \(k\), that is, whether there exists a graph \(G\) such
that \(\alpha_k'(G) > \vartheta_k(G)\).
The following result shows that this is not the case.

\begin{theorem}
  For every graph \(G\) and every real \(k \geq 1\),
  \[
    \alpha_k'(G) \leq \vartheta_k(G).
  \]
\end{theorem}
\begin{proof}
  Let \(H\) be an induced subgraph of \(G\) with
  \(|V(H)| = \alpha_k'(G)\) and \(\chi_f(H) \leq k\), and let
  \(y \in \Reals_+^{\cS(H)}\) be an optimal solution for the
  LP~\cref{eq:7} defining \(\chi_f(H)\).
  Define \(\hat{X} \in \Sym{V}\) by
  \[
    \hat{X} \coloneqq \sum_{S \in \cS(H)} y_S \frac{\ones_S^{} \ones_S^\T}{|S|}.
  \]
  It follows immediately that \(\hat{X} \succeq 0\),
  \(A_G \circ \hat{X} = 0\), and
  \(\iprod{I}{\hat{X}} = \sum_{S \in \cS(H)} y_S = \chi_f(H) \leq k\).
  From the LP formulation,
  \[ \hat{X} ~ \ones_{V(H)} = \sum_{S \in \cS(H)} y_S \ones_S =
    \ones_{V(H)},\] and because \(\hat{X}\) is a nonnegative matrix, the
  Perron-Frobenius theorem (see for instance
  \cite[Section~2.2]{BrouwerH12a}) implies that the largest
  eigenvalue of \(\hat{X}\) is \(1\), hence \(\hat{X} \preccurlyeq I\).
  Finally,
  \[
    \ip{J}{\hat{X}} = \sum_{S \in \cS(H)} y_S |S| =
    (\ones_{V(H)})^\top \left(\sum_{S \in \cS(H)} y_S \ones_S\right) =
    |V(H)| = \alpha_k'(G).
  \]
  Hence if we let \(\ell \coloneqq \chi_f(H)\), then \(\hat{X}\) is
  feasible for the SDP \eqref{eq:thkdef} defining
  \(\vartheta_\ell(G)\) with objective value \(\alpha_k'(G)\).
  Combining with monotonicity of \(\vartheta_k(G)\) with respect
  to~\(k\), we obtain
  \(\alpha_k'(G) \leq \vartheta_\ell(G) \leq \vartheta_k(G)\) since
  \(\ell \leq k\).
\end{proof}

\section{The Weighted \texorpdfstring{\(\vartheta_k(G,\cdot)\)}{ϑₖ(G,·)} Function}
\label{sec:weights}

The Lovász theta graph parameter was extended to a real-valued
function over vectors of nonnegative vertex weights in
\cite{GroetschelLS86a}.
This function was shown to be a positive definite monotone gauge, and
this led to a rich theory connecting polyhedral combinatorics and
convex optimization (see \cite{BenedettoCC21a} for a self-contained
review).
\Nameandcite{Alizadeh95a} introduced a vertex-weighted version of
\(\vartheta_k\).
In this section we briefly review known results and basic properties,
connect them to a newly introduced vertex weighted version of
\(\varphi_k\), show a connection between these and \(\chi_f\), but
more notably we show that \(w \mapsto \vartheta_k(G,w)\) is not a
gauge.
While a negative result is perhaps not exciting, we expect it to
clarify the geometric limitations of \(\vartheta_k(G,\cdot)\) as a
relaxation, and to guide the development of alternative formulations
that retain desirable convex–analytic properties, such as gauge or
norm structures.

Throughout this section, \(w \in \Reals_+^V\) denotes a nonnegative
vector of vertex weights.
Define \(\sqrt{w} \in \Reals_+^V\) as the vector obtained by taking the
entry-wise square root, and
\begin{equation*}
  W \coloneqq \sqrt{w}\sqrt{w}^{\T}.
\end{equation*}
We refer to the Master's dissertation~\cite{Oliveira24a} of the
third author for a detailed treatment of the results in this section.

\subsection{A sandwich theorem and connection to \texorpdfstring{\(\chi_f\)}{χ\_f}}

\Nameandcite{Alizadeh95a} defined
\[
  \vartheta_k(G,w) \coloneqq
  \max\setmap[\big]{\ip{W}{X}}{
    \ip{I}{X} = k,\, A_G \schur X = 0,\, 0 \preceq X \preceq I
  }.
\]
Here we introduce three definitions in order to obtain a
vertex-weighted sandwich theorem generalizing
\(\alpha_k(G) \leq \varphi_k(G) \leq \vartheta_k(G) \leq \chi_k(\overline G)\).

Define
\begin{equation}
  \label{def:varphiw}
  \varphi_k(G,w)
  \coloneqq
  \max \left\{ \, \ip{M}{Y} : M,Y \in \Psd{V}, \, \diag Y = w,\,
    A_G \circ Y = 0, \, \ip{I}{M} = k,\,
    M \preccurlyeq I \right\}.
\end{equation}
On the combinatorial side, for each integer \(k\), set
\[
  \alpha_k(G,w) \coloneq
  \max\left\{\sum_{i\in U}w_i : U\subseteq V(G),\ G[U]\ \text{is \(k\)-colorable}\right\}.
\]
Recall that \(\cS(G)\) denotes the set of nonempty cocliques of \(G\).
Define, for each \(w \in \Integers_+^V\),
\[
  \chi_k(G,w) \coloneqq\min\left\{\sum_{S\in\cS(G)}y_S :
    y\in\Integers_+^{\cS(G)},\ \sum_{S\in\cS(G)}y_S\ones_S = k\,w\right\}.
\]

The best argument for why these are the correct generalizations is the
theorem below and its proof, which is a natural generalization of
standard arguments both in the case \(k=1\) for general weights, and
in the case \(k > 1\) for \(w = \ones\).
Because the proof does not bring any new relevant insight and because
it is slightly technical in dealing with corner cases (vertices for
which \(w_i = 0\), for instance), we prefer to omit it here, and refer
the reader to \cite[Theorem 5.2]{Oliveira24a}.

\begin{theorem}
  \label{thm:weighted-sandwich}
  Let \(G = (V,E)\) be a graph and let \(k\) be a nonnegative integer.
  Then for every \(w \in \Reals_+^V\),
  \[
    \alpha_k(G,w) \leq \varphi_k(G,w) \leq \vartheta_k(G,w).
  \]
  The second inequality also holds for nonnegative real \(k\).
  If \(w \in \Integers_+^V\) and \(k \in \Integers_+\), then
  \[
    \vartheta_k(G,w) \leq \chi_k(\overline G,w).
  \]
\end{theorem}

We introduce a vertex-weighted version of the parameter \(\mu\)
studied earlier:
\begin{equation}
  \label{eq:muw}
  \mu(G,w) \coloneqq \min\{k \in \Reals_+ : \vartheta_k(G,w)=\ip{w}{\ones}\}.
\end{equation}
In \cite[Theorem 4]{CarliCG21a} (see also
\cite[Proposition~4.7]{Oliveira24a}), the following result is
proved.
\begin{theorem}
  Let \(G=(V,E)\) be a graph. For all \(w \in \Reals_+^V\),
  \[
    \mu(G,w) \leq \chi_f(G).
  \]
\end{theorem}

Since \(\mu(G) = \mu(G,\ones)\), the above result can be used to
separate the inequality~\cref{eq:5}.
Namely, by taking \(G\) to be the 5-cycle, we get
\(\mu(G) \leq \chi_f(G) < \xi(\overline G)\).
We refer the reader to \cite[Proposition 5.13]{Oliveira24a} for
further details.

\subsection{Parameter \texorpdfstring{\(\vartheta_k(G,\cdot)\)}{𝜗ₖ(G,·)} is not a gauge}

A function \(\kappa\colon \Reals_+^V \to \Reals\) is a positive definite
gauge if \(\kappa\) is nonnegative, zero-valued only at the origin,
positively homogeneous, and subadditive.
That is, \(\kappa(x) \geq 0\) for each \(x \in \Reals_+^V\), with
equality if and only if \(x = 0\),
\(\kappa(\lambda x) = \lambda\kappa(x)\) for each
\(\lambda \in \Reals_+\) and \(x \in \Reals_+^V\), and
\(\kappa(x+y) \leq \kappa(x)+\kappa(y)\) for each
\(x,y\in \Reals_+^V\).
We call \(\kappa\) monotone if \(\kappa(x) \leq \kappa(y)\) whenever
\(x \leq y\) (entry-wise).
Positive definite monotone gauges are essentially the same as norms
restricted to the nonnegative orthant, and they are a fundamental tool
in combinatorial optimization.
See \cite{BenedettoCC21a} for a self-contained review on gauges and
their connection to combinatorial optimization.

For \(k=1\), the map \(w \mapsto \vartheta_1(G,w)\) is a positive
definite monotone gauge.
Its unit convex corner, that is,
\(\setst{w\in\Reals_+^V}{\vartheta_1(G,w)\leq1}\), is the theta body
of the complement graph, and this has been extensively studied since
its introduction in \cite{GroetschelLS86a}.

For \(k>1\), this fails in general.

\begin{proposition}
  \label{prop:nonconvex}
  There are a graph \(G\), an integer \(k>1\), and weights
  \(w,z \in \Integers_+^V\) such that

  \[
    \vartheta_k(G,w+z) > \vartheta_k(G,w)+\vartheta_k(G,z).
  \]
\end{proposition}
\begin{proof}
  Take \(G\) to be the disjoint union \(K_3 \cup K_1\),
  \(k \coloneqq 2\), and \(w\coloneqq\ones\).
  Define \(z\) as having weight 1 on the vertices of \(K_3\) and
  weight 3 on the single vertex of \(K_1\).
  From the explicit formula for \(\vartheta_k(K_n\cup K_1)\) proved in
  \cite{Oliveira24a}, we have
  \[
    \vartheta_2(G,w)=2+\sqrt{3}.
  \]
  Consider
  \[
    X\coloneqq\frac12
    \begin{bmatrix}
      1 & \frac1{\sqrt3} & \frac1{\sqrt3} & \frac1{\sqrt3}\\
      \frac1{\sqrt3} & 1 & 0 & 0\\
      \frac1{\sqrt3} & 0 & 1 & 0\\
      \frac1{\sqrt3} & 0 & 0 & 1
    \end{bmatrix}.
  \]
  This matrix is feasible for \(\vartheta_2(G,z)\) and has objective value
  \[
    \left\langle \sqrt{z}\sqrt{z}^\T, X\right\rangle = 6 = \ip{z}{\ones},
  \]
  so \(\vartheta_2(G,z)=6\).
  Now evaluate the same feasible matrix for \(w+z\):
  \[
    \vartheta_2(G,w+z) \geq \left\langle \sqrt{w+z}\sqrt{w+z}^\T, X\right\rangle
    = 5 + 2\sqrt{6}.
  \]
  Hence
  \[
    \vartheta_2(G,w+z)
    \ge 5+2\sqrt6
    > 8+\sqrt3
    = \vartheta_2(G,w)+\vartheta_2(G,z).\qedhere
  \]
\end{proof}

As a consequence, one cannot use \(\vartheta_k(G,\cdot)\) the same way
\(\vartheta(G,\cdot)\) is used to define an analogous version of a
\(k\)-fold theta body.
We discuss in the next section possible alternative approaches to
define a \(k\)-fold theta body.

\subsection{The parameter \texorpdfstring{\(\varphi_k(G,\cdot)\)}{φₖ(G,·)} is a gauge}

The parameter \(\varphi_k(G,\cdot)\) retains the gauge structure for
\(k > 1\).

\begin{theorem}
  Let \(G\) be a graph and let \(k>0\).
  Then the map \(w \mapsto \varphi_k(G,w)\) defined in
  \eqref{def:varphiw} is a positive definite monotone gauge.
\end{theorem}
\begin{proof}
  Let \(w,z \in \Reals_+^V\) and let \(\lambda \geq 0\).
  We will show that \(\varphi_k(G,\cdot)\) is positive definite,
  positively homogeneous, subadditive, and monotone.

  \begin{description}[leftmargin=*,itemsep=3pt]
  \item[Positive definiteness] the facts that \(\varphi_k(G,\cdot)\)
    is nonnegative and \(\varphi_k(G,0) = 0\) are immediate.
    If \(w \neq 0\), then any matrix \(Y\) with \(\diag Y = w\) has a
    positive eigenvalue.
    By choosing \(M\) to contain a positive eigenvalue in the same
    eigenspace, we can ensure that \(\ip{M}{Y} > 0\), hence
    \(\varphi_k(G,w) > 0\).

  \item[Positive homogeneity] let \((\hat{M},\hat{Y})\) be feasible
    for \(\varphi_k(G,w)\) in~\cref{def:varphiw}, attaining the optimum.
    Note that \((\hat{M},\lambda\hat{Y})\) is feasible for
    \(\varphi_k(G,\lambda w)\) with objective value
    \(\lambda \varphi_k(G,w)\).
    Thus, we have
    \(\varphi_k(G,\lambda w) \geq \lambda \varphi_k(G,w)\).
    On the other hand, if \((\hat{M},\hat{Y})\) is feasible for
    \(\varphi_k(G,\lambda w)\), with $\lambda \neq 0$, and attaining the optimum, then
    \((\hat{M},\tfrac{1}{\lambda}\hat{Y})\) is feasible for
    \(\varphi_k(G,w)\) and its objective value is
    \(\tfrac{1}{\lambda}\varphi_k(G,\lambda w)\).
    Thus \(\varphi_k(G,\lambda w) = \lambda \varphi_k(G,w)\).
  \end{description}

  To see subadditivity, we resort to an alternative formulation of
  \(\varphi_k(G,w)\).
  Let us define the set
  \[
    \texttt{M}_k(G) \coloneqq
    \left\{x \in \Reals_+^V :
      \begin{array}{l}
        \exists~M \in \Psd{V}, M \preccurlyeq I, \ip{I}{M} = k,\\
        \exists~u \colon V \to \Reals^V \text{ an orthonormal representation of } \overline G\\
        \text{with } x_i = u_i^\T M u_i
      \end{array}\right\}.
  \]
  An immediate extension of Theorem~\ref{thm:varphi-orth-rep} shows that
  \[
    \varphi_k(G,w) = \max\{\ip{w}{x} : x \in \texttt{M}_k(G)\}.
  \]
  (See \cite[Theorem 5.11]{Oliveira24a} for a detailed proof of
  this equivalence.)
  \begin{description}[leftmargin=*,itemsep=3pt]
  \item[Subadditivity] follows from
    \begin{align*}
      \varphi_k(G,w+z) & = \max\{\ip{w+z}{x} : x \in \texttt{M}_k(G)\} \\
                       & = \max\{\ip{w}{x} + \ip{z}{x} : x \in \texttt{M}_k(G)\} \\
                       & \leq \max\{\ip{w}{x} : x \in \texttt{M}_k(G)\} + \max\{\ip{z}{x} : x \in \texttt{M}_k(G)\} \\
      & = \varphi_k(G,w) + \varphi_k(G,z).
    \end{align*}
  \item[Monotonicity] if \(w \leq z\), then \(\ip{w}{x} \leq \ip{z}{x}\)
    for every \(x \in \texttt{M}_k(G)\).
    Hence
    \(\varphi_k(G,w) = \max\{\ip{w}{x} : x \in \texttt{M}_k(G)\} \leq
    \max\{\ip{z}{x} : x \in \texttt{M}_k(G)\} = \varphi_k(G,z)\).
  \end{description}
\end{proof}

While the result above offers a strong perspective on the extension of
the \(k = 1\) theory for general~\(k\), we do not know whether the set
\(\texttt{M}_k(G)\) is SDP representable for \(k > 1\).
In fact, it is not even clear whether \(\texttt{M}_k(G)\) is convex
for \(k > 1\).
This is the motivation for Section~\ref{sec:theta-body-k}.

\subsection{A Greene--Kleitman--Lovász upper bound to
  \texorpdfstring{\(\varphi_k(G,w)\)}{φₖ(G,w)}}
\label{subsec:psi-k}

\Nameandcite{Greene76a} and
Greene--Kleitman~\cite{GreeneK76a} studied a colouring parameter that
Lovász later used in his definition of \(k\)-perfect graphs.
For a graph \(G=(V,E)\), let
\[ \psi_k(G) = \min_{S\subseteq V} \{ k\chi(G-S) + |S| \}.
\]
If \(G\) is a comparability graph, then cliques are chains and
\[
  \min_{\mathcal{C}} \sum_{C \in
    \mathcal{C}}\min\set[\big]{\card{C},k}
  =
  \min_{S\subseteq V} \set[\Big]{k\chi(\overline G-S)+\card{S}},
\]
where the leftmost minimum is taken over all chain
partitions~\(\mathcal{C}\).

We proceed to define the fractional weighted version for a vector
\(w \in \Reals_+^V\) and a positive integer \(k\) by
\[
  \psi^f_k(G,w)
  \coloneqq
  \min_{S\subseteq V}
  \left\{
    k~\chi_f(G-S,w_{V\setminus S})+
    \sum_{i\in S}w_i
  \right\}.
\]
The parameter interpolates between colouring and vertex deletion: one
may delete a set \(S\), paying its full weight, and then pay \(k\)
times the (fractional) colouring cost of the remaining graph.
Lovász's class of \(k\)-perfect graphs is motivated by equality
between \(\alpha_k\) and the ordinary version of this parameter on all
induced subgraphs.

\begin{theorem}
  \label{thm:varphi-psi}
  Let \(G=(V,E)\) be a graph on \(n\) vertices, let \(w\in\Reals_+^V\), and let
  \(k\in[n]\).
  Then
  \[
    \varphi_k(G,w)\leq \psi^f_k(\overline G,w).
  \]
\end{theorem}

\begin{proof}
  Fix \(S\subseteq V\).
  We split the proof in two cases.
  If \(k > n - |S|\), then
  \begin{align*}
    \varphi_k(G,w)
    &\leq w^\T \ones \\
    & \leq (n-|S|)~\max\{w_i : i \in V\setminus S\}  + w^\T \ones_S \\
    & \leq k~\chi_f(\overline G-S,w_{V\setminus S}) + w^\T \ones_S.
  \end{align*}
  Otherwise \(k \leq n-|S|\).
  Let \(\widehat{Y}\) be optimum for \eqref{def:varphiw}, and
  \(\widehat{Y}_{\overline{S}}\) its principal submatrix corresponding
  to \(V \setminus S\).
  By eigenvalue interlacing applied to a principal submatrix, one has
  \[ \trace \widehat{Y}_{\overline{S}} - \sum_{i = 1}^k
    \lambda_i(\widehat{Y}_{\overline{S}})
    \leq \trace \widehat{Y} - \sum_{i = 1}^k \lambda_i(\widehat{Y}),
  \]
  thus
  \begin{align*}
  \varphi_k(G,w) & \leq \sum_{i = 1}^k
                   \lambda_i(\widehat{Y}_{\overline{S}})
                   + (\trace \widehat{Y} - \trace \widehat{Y}_{\overline{S}}) \\
  & \leq \varphi_k(G - S,w_{V\setminus S}) + w^\T \ones_S.
  \end{align*}
  Also \(\varphi_k(H,u)\leq k\varphi_1(H,u)\) for every graph \(H\) and
  \(u \in \Reals_+^{V(H)}\).
  Since \(\varphi_1(H,u)=\vartheta(H,u)\leq\chi_f(\overline H,u)\), we get
  \[
    \varphi_k(G,w)
    \leq
    k~\chi_f(\overline G-S,w_{V\setminus S})+
    \sum_{i\in S}w_i .
  \]
  Taking the minimum over all \(S\subseteq V\) gives the result.
\end{proof}

For \(G \coloneqq \overline {K_n}\), we have
\(\psi_k^f(\overline G) = \psi_k(\overline G) = n\) and \(\chi_k(\overline G) = kn\).
Thus, \cref{thm:varphi-psi} is sharper than~\cref{eq:k-sandwich} on
some graphs.

\section{The \texorpdfstring{\(k\)}{k}-fold theta body}
\label{sec:theta-body-k}

The failure of subadditivity for \(\vartheta_k(G,\cdot)\) prevents us
from using \(\vartheta_k\) directly to define a convex theta body in
the same way that \(\vartheta(G,\cdot)\) can be used to define
\(\THbody(G)\).
We therefore introduce an explicitly convex substitute.
For an integer \(k\geq1\), our natural proposal for a \(k\)-fold theta
body is given by
\begin{equation}
  \THbody_k(G) \coloneqq
  \left\{x \in \Reals_+^V :
    \begin{array}{l}
      x = x^{(1)} + \cdots + x^{(k)}, \\
      x^{(i)} \in \THbody(G) \text{ for each } i \in[k], \\
      0 \leq x \leq \ones .
    \end{array}\right\}.
\end{equation}
This can be naturally extended to nonnegative real \(k\), recalling
that \(\THbody(G)\) is a convex set, by
\begin{equation}
  \label{def:THk}
  \THbody_k(G) \coloneqq k \THbody(G) \cap [0,1]^V.
\end{equation}

Note that this is a convex set.
We then naturally define, for each \(w \in \Reals_+^V\),
\begin{equation}
  \label{def:thkprime}
  \theta_k(G,w) \coloneqq \max\setst[\big]{\ip{w}{x}}{x \in \THbody_k(G)}
\end{equation}

The following proposition is immediate.
\begin{proposition}
  The function \(\theta_k(G,\cdot)\) defined in \eqref{def:thkprime} is
  a positive definite monotone gauge.
\end{proposition}

We can also show that \(\theta_k(G,w)\) is an upper bound for \(\alpha_k(G,w)\):
\begin{proposition}
  For every graph \(G = (V,E)\), every positive integer \(k\), and every
  \(w \in \Reals_+^V\),
  \[\alpha_k(G,w) \leq \theta_k(G,w).\]
\end{proposition}
\begin{proof}
  This follows from the fact that \(\THbody(G)\) contains the
  characteristic vector of every stable set of \(G\), hence
  \(\THbody_k(G)\) contains the characteristic vector of every
  \(k\)\nobreakdash-colorable induced subgraph of \(G\).
\end{proof}

We now show that \(\theta_k(G,w)\) can be computed as the optimal
value of an SDP.
We highlight that a very similar SDP was proposed by
\citeauthor{KuryatnikovaSV22a} in \cite{KuryatnikovaSV22a}
with an extra constraint that \(Z \geq 0\), though without the
motivation of defining a \(k\)-fold theta body.

\begin{theorem}
  Let \(G=(V,E)\) be a graph and let \(k>0\).
  Then
  \begin{equation}
    \label{eq:8}
    \THbody_k(G)
    =
    \left\{ \diag (Z) :
      \begin{array}{l}
        Z\in \Sym{V},\,
        A_G \circ Z = 0,\,
        I \circ Z \leq I,\\[2mm]
        \begin{pmatrix}
          k & \diag(Z)^\T\\
          \diag(Z) & Z
        \end{pmatrix}
                     \succeq 0
      \end{array}
    \right\}.
  \end{equation}
\end{theorem}

\begin{proof}
  Recall that \(x\in \THbody(G)\) if and only if there exists
  \(Y\in \Sym{V}\) such that \(\diag(Y)=x\), \(A_G \circ Y = 0\), and
  \[
    \begin{pmatrix}
      1 & x^\T\\
      x & Y
    \end{pmatrix}
    \succeq 0.
  \]
  Hence \(x\in k\,\THbody(G)\) if and only if \(\tfrac1k x\in
  \THbody(G)\). By the previous characterization, this holds if
  and only if there exists \(Y\in \Sym{V}\) such that
  \(\diag(Y)=\tfrac1k x\), \(A_G \circ Y = 0\), and
  \[
    \begin{pmatrix}
      1 & \tfrac1k x^\T\\
      \tfrac1k x & Y
    \end{pmatrix}
    \succeq 0.
  \]
  Multiplying this block matrix by \(k\), and setting
  \(Z \coloneqq kY\), we get the equivalent conditions
  \(\diag(Z) = x\), \(A_G \circ Z = 0\), and
  \[
    \begin{pmatrix}
      k & x^\T\\
      x & Z
    \end{pmatrix}
    \succeq 0.
  \]
  Therefore
  \[
    k\,\THbody(G)
    =
    \left\{
      \diag (Z) :
      Z\in \Sym{V},\,
      A_G \circ Z = 0,\,
      \begin{pmatrix}
        k & \diag(Z)^\T\\
        \diag(Z) & Z
      \end{pmatrix}
      \succeq 0
    \right\}.
  \]

  By definition,
  \(
    \THbody_k(G) = k\,\THbody(G)\cap [0,1]^V.
  \)
  Since \(x=\diag (Z)\), the additional condition \(x\leq \ones\) is
  exactly
  \(
    I \circ Z \leq I.
  \)
  The condition \(x\geq 0\) is automatic from the positive
  semidefiniteness of the block matrix, because \(Z\succeq 0\) and
  hence \(Z_{ii}\geq 0\) for all \(i\in V\).

  Combining these observations gives the desired representation.
\end{proof}

Our main result in this section is that \(\theta_k(G,w)\) is an upper
bound for \(\varphi_k(G,w)\).

\begin{theorem}
  Let \(G=(V,E)\) be a graph, let \(k\in \Reals_+\), and let
  \(w\in\Reals_+^V\). Then
  \[
    \varphi_k(G,w)\leq \theta_k(G,w).
  \]
\end{theorem}

\begin{proof}
  We use the orthonormal representation formulation of
  \(\varphi_k(G,w)\). Namely,
  \[
    \varphi_k(G,w)
    = \max \left\{ \sum_{i\in V} w_i u_i^\T M u_i :
      \begin{array}{l}
        u \text{ is an orthonormal representation of } \overline G,\\
        0\preccurlyeq M\preccurlyeq I,\ \ip{I}{M}=k
      \end{array}
    \right\}.
  \]
  Let \(u\) and \(M\) be feasible for this formulation. Define
  \(x\in\Reals^V\) by
  \[
    x_i \coloneqq u_i^\T M u_i \qquad \forall i\in V.
  \]
  We claim that \(x\in \THbody_k(G)\).

  Since \(0\preccurlyeq M\preccurlyeq I\), take a spectral
  decomposition
  \[
    M=\sum_{r=1}^m t_r c_r c_r^\T,
  \]
  where \(c_1,\ldots,c_m\) are orthonormal vectors,
  \(0\leq t_r\leq 1\), and
  \[
    \sum_{r=1}^m t_r=\ip{I}{M}=k.
  \]
  For each \(r\), define \(x^{(r)}\in\Reals^V\) by
  \[
    x_i^{(r)} \coloneqq\langle c_r,u_i\rangle^2.
  \]
  We first show that \(x^{(r)}\in\THbody(G)\). Indeed, consider
  the vectors
  \[
    g_0 \coloneqq c_r, \qquad
    g_i \coloneqq \langle c_r,u_i\rangle u_i \quad \forall i\in V.
  \]
  Their Gram matrix \(Z^{(r)}\), indexed by \(\{0\}\cup V\), satisfies
  \[
    Z^{(r)}_{00}=1, \qquad
    Z^{(r)}_{0i}=x_i^{(r)}=Z^{(r)}_{ii} \quad \forall i\in V.
  \]
  Moreover, since \(u\) is an orthonormal representation of \(\overline G\),
  we have \(\langle u_i,u_j\rangle=0\) whenever \(ij\in E(G)\). Hence
  \[
    Z^{(r)}_{ij}
    = \langle c_r,u_i\rangle \langle c_r,u_j\rangle \langle u_i,u_j\rangle = 0 \qquad \forall ij\in E(G).
  \]
  Thus
  \(x^{(r)}\in\THbody(G)\).

  Now
  \[
    x_i = u_i^\T M u_i = \sum_{r=1}^m t_r \langle c_r,u_i\rangle^2 = \sum_{r=1}^m t_r x_i^{(r)}.
  \]
  Since \(\sum_r t_r=k\) and \(\THbody(G)\) is convex, it follows
  that
  \[
    x\in k\,\THbody(G).
  \]
  Also, because \(0\preccurlyeq M\preccurlyeq I\) and
  \(\|u_i\|=1\), we have
  \[
    0\leq x_i=u_i^\T M u_i\leq u_i^\T u_i=1
    \qquad \forall i\in V.
  \]
  Therefore
  \[
    x\in k\,\THbody(G)\cap[0,1]^V = \THbody_k(G).
  \]

  Finally,
  \[
    \sum_{i\in V} w_i u_i^\T M u_i = \sum_{i\in V} w_i x_i =
    \ip{w}{x} \leq \max\{\ip{w}{y}:y\in\THbody_k(G)\} =
    \theta_k(G,w).
  \]
  Since this holds for every feasible pair \((u,M)\), we conclude that
  \[
    \varphi_k(G,w)\leq \theta_k(G,w).\qedhere
  \]
\end{proof}

As usual, define \(\theta_k(G) \coloneqq \theta_k(G,\ones)\).
\begin{proposition}
  \label{prop:varphi-strict-theta-k}
  There exist a graph \(G\) and an integer \(k\) such that
  \[
    \varphi_k(G)<\theta_k(G).
  \]
\end{proposition}

\begin{proof}
  Let \(G \coloneqq KG(6,2)\), the Kneser graph whose vertices are the
  2-subsets of \([6]\), with two vertices adjacent when the
  corresponding sets are disjoint.
  Then \(\vartheta(G) = 5\); see~\cite{Lovasz79a}.

  Averaging an optimal point of \(\THbody(G)\) over the
  (vertex-transitive) automorphism group gives
  \(\tfrac13\ones\in\THbody(G)\).
  Thus \(\ones\in 3\THbody(G)\cap[0,1]^V=\THbody_3(G)\), and hence
  \[
    \theta_3(G)=15.
  \]
  On the other hand, \(\xi(\overline G)=4\) (see \cite{Haviv19a}).
  From Theorem~\ref{thm:varphi-orth-rank}, it follows that
  \[
    \varphi_3(G)<15=\theta_3(G).\qedhere
  \]
\end{proof}

The next result can be seen as a \(\theta_k\) counterpart to
\cref{thm:2}, albeit with a stronger hypothesis.
We denote \(\theta \coloneqq \theta_1 = \vartheta\).
\begin{corollary}
  Let \(G\) be a vertex-transitive graph on \(n\) vertices.
  Let \(k \geq 1\) be a real number.
  Then
  \[
    \theta(G \Cartesian K_k) = \min\set{k\theta(G),n} = \theta_k(G).
  \]
\end{corollary}
\begin{proof}
  Averaging an optimal solution attaining \(\theta(G)\)
  (on~\(\THbody(G)\)) over the automorphism group of~\(G\), yields
  \begin{equation*}
    \frac{\theta(G)}{n} \ones \in \THbody(G).
  \end{equation*}
  Since \(\THbody(G)\) is a convex corner, \(t\ones \in \THbody_k(G)\)
  for \(t \coloneqq \min\set{k\theta(G)/n,1}\).
  Thus,
  \(\theta_k(G) = \theta_k(G,\ones) \geq \min\set{k\theta(G),n}\).
  The reverse inequality follows from
  \(\theta_k(G,w) \leq
  \min\set[\big]{k\theta(G,w),\iprod{w}{\ones}}\).
  Applying \cref{thm:2} concludes the proof.
\end{proof}

\section{Conclusion}

We studied several theta-type relaxations for the maximum
\(k\)-colorable induced subgraph problem.
The Narasimhan--Manber parameter \(\vartheta_k\) is the most direct
SDP generalization of Lov\'asz theta, but it is neither subadditive as
a function of vertex weights, nor naturally related to orthonormal
representations.
In this paper we filled this gap in the literature by exploring
parameters that satisfy these roles and we proved several of their
properties, leading to connections with graph products, regularity,
orthonormal representations, and convex corners.

Here is a list of questions that remain open.

\subsection*{Open questions}

\begin{enumerate}[label=(\roman*),leftmargin=*]
  \item Compare \(\theta_k\) and \(\vartheta_k\).
    Is it true that
  \[
    \theta_k(G,w)\leq \vartheta_k(G,w)
    \qquad\forall G,k,w\geq0?
  \]
  Since \(w\mapsto\vartheta_k(G,w)\) is not sublinear in general, a
  more interesting question is obtained by letting
  \(\vartheta_k^{\conv}(G,\cdot)\) be the largest monotone gauge
  dominated by \(w\mapsto\vartheta_k(G,w)\).  Does
  \[
    \vartheta_k^{\conv}(G,w)\leq \theta_k(G,w)
  \]
  hold for all \(G,k,w\geq0\)?
  Can equality hold, perhaps under additional symmetry assumptions?

\item[] \Nameandcite{KuryatnikovaSV22a} define a matrix-lifting
  relaxation \(\theta_k^3\), which is essentially the Schrijver-type
  version of the body \(\THbody_k(G)\), obtained by adding entrywise
  nonnegativity to the SDP with the RHS of~\cref{eq:8} as its feasible
  region.
  They conjecture that \(\theta_k^3(G)\leq\vartheta_k'(G)\), which is
  the Schrijver-type version of our question above.

\item \Citeauthor{SinjorgoS22a} conjectured that
  \[
    \vartheta(G \Cartesian K_k)\leq\vartheta_k(G),
  \]
  with equality whenever \(\vartheta_k(G)=k\vartheta(G)\).
  We showed this for a broad homogeneous class of graphs.
  Does the conjecture hold for all graphs and all real \(k\in[1,n]\)?

\item Is
  \[
    \mu(G)
    \leq
    \vartheta(\overline{G})
  \]
  for every graph \(G\)?
  Numerical evidence suggests that equality can fail, but no example
  is known to us with \(\mu(G)>\vartheta(\overline{G})\).

\item Is the orthonormal representation set
  \(\texttt{M}_k(G)\) convex?  If not, can one characterize its closed
  convex lower hull, i.e., the convex corner whose support function is
  \(\varphi_k(G,\cdot)\)?
\end{enumerate}

\section*{Acknowledgments}

We would like to acknowledge Rafael Grandsire and Renata Sotirov for
comments and suggestions about this work.

\section*{Statement of AI Use}

Generative AI (ChatGPT 5.5) was used for organizing and reviewing the
text.
It was also used to find the example in
Proposition~\ref{prop:varphi-strict-theta-k}.
The authors are fully responsible for the entire content of this
paper.

\printbibliography

\end{document}